\documentclass[12pt]{amsart}

\usepackage{cite}

\usepackage{amsfonts}
\usepackage{amssymb, latexsym, amsthm, amsmath, verbatim}

\theoremstyle{definition}
\newtheorem{Thm}{Theorem}[section]
\newtheorem{Prop}[Thm]{Proposition}
\newtheorem{Cor}[Thm]{Corollary}
\newtheorem{Lem}[Thm]{Lemma}

\newtheorem{Rmk}[Thm]{Remark}

\makeatletter
\def\imod#1{\allowbreak\mkern5mu{\operator@font mod}\,\,#1}
\makeatother

\allowdisplaybreaks[4]

\begin{document}

\title[On the Non-vanishing of Hilbert Poincar\'e Series]{On the Non-vanishing of Hilbert Poincar\'e Series}
\author{Mingkuan Zhang and Yichao Zhang$^\star$}
\address{School of Mathematics, Harbin Institute of Technology, Harbin 150001, P. R. China}
\email{19b912043@stu.hit.edu.cn}
\address{Institute for Advanced Study in Mathematics of HIT and School of Mathematics, Harbin Institute of Technology, Harbin 150001, P. R. China}
\email{yichao.zhang@hit.edu.cn}
\date{}
\thanks{$^{\star}$ partially supported by a grant of National Natural Science Foundation of China (no. 11871175)}

\begin{abstract} We prove that if $\nu$ has small norm with respect to the level and the weight, the $\nu$-th Hilbert Poincar\'e series does not vanish identically. We also prove Selberg's identity on Kloosterman sums in the case of number fields, which implies certain vanishing and non-vanishing relations of Hilbert Poincar\'e series when the narrow class number is $1$. Finally, we pass to the adelic setting and interpret the problem via Hecke operators.
\end{abstract}

\maketitle
\newcommand{\Z}{{\mathbb Z}} 
\newcommand{\Q}{{\mathbb Q}} 

\section{Introduction and Main Results}
It is conjectured that the Poincar\'e series $P_{k,m}$ of level one and weight $k$ does not vanish identically for each positive integer $m$, provided that the corresponding space of cuspforms is not trivial. This is known as the Poincar\'e Conjecture on Poincar\'e series. For fixed positive integers $m$ and $k$, it is well-known that $P_{k,m}$ vanishes identically if and only if the $m$-th Hecke operator $T_m$ of weight $k$ vanishes identically. As a special case, when $k=12$, the Poincar\'e Conjecure is equivalent to the famous Lehmer's conjecture \cite{L}, the latter of which claims that Ramanujan's $\tau$-function is non-vanishing everywhere. It is remarkable that the irrationality of the coefficients of certain harmonic Maass forms implies the non-vanishing of $\tau(m)$, as shown by Ono in \cite{20}. Lehmer \cite{L} proved that the smallest positive integer $m$ with $\tau(m)=0$ is a prime number, and Serre \cite{19} obtained the first asymptotic upper bound for the number of primes with $\tau(p)=0$ and raised the problem of seeking more precise statements. For more results along this line of research, see \cite{B} and \cite{K1} and the references therein. 

By estimating the $J$-Bessel function and the Kloosterman sums, Rankin \cite{R} obtained an important result towards the Poincar\'e Conjecture: there exist constants $B>4 \log 2$ and $k_{0}$ such that the Poincar\'e series $P_{k,m}$ does not vanish identically whenever
$$m \leq k^{2} \exp \left(\frac{-B \log k}{\log \log k}\right)\quad \text{and}\quad k \geq k_{0}.$$ Later, Lehner \cite{Lehn} and Mozzochi \cite{7} generalized Rankin's result to more general Fuchsian groups. For any subgroup of finite index in $\text{SL}_{2}(\mathbb{Z})$, using Weil's estimate for the Kloosterman sums,  Gaigalas \cite{G} showed that for any fixed positive integer $m$, there exist infinitely many weights $k$ such that the $m$-th Poincar\'e series of weight $k$ is not identically zero. 

It is natural to consider the Poincar\'e Conjecture for Hilbert Poincar\'e series, passing from the elliptic case to the case of Hilbert modular forms. Kumari \cite{9} gave an answer to this question in both of the weight aspect and the level aspect. For example, when the level is fixed, he proved that the $\nu$-th Hilbert Poincar\'e series does not vanish identically for large enough weight. Note that in a particular case this is a generalization of Gaigalas's result to Hilbert Poincar\'e series \cite{G}. The key ingredient in Kumari's proof is that the Fourier coefficients of Hilbert Poincar\'e series satisfy certain orthogonality relations with respect to both of the weight and the level. Such orthogonality generalizes the work of Kowalski, Saha and Tsimerman \cite{K} and is of independent interest. 

Let $k\geq 4$ be even. In this paper, following Rankin's original treatment in \cite{R}, we first work on the generalized Kloosterman sums and the $J$-Bessel function as in \cite{10} to estimate the Fourier coefficients of Hilbert Poincar\'e series, and then obtain our non-vanishing condition on Hilbert Poincar\'e series. Explicitly, for a nonzero integral ideal $\mathfrak{c}$, we prove that if $\mu\in\mathfrak{c}^+$ satisfies 
$$|N(\mu)| < C(k-1)^{\frac{2nk-2n}{2k-1}} N(\mathfrak{cn})^{\frac{2k-3}{2k-1}}$$
with $C$ being an effective positive constant depending on $F$ only, then the $\mu$-th Hilbert Poincar\'e series $P_{\mu}(z ; k, \mathfrak{c},\mathfrak{n})$ of weight $k$ and level $\mathfrak{n}$ is not identically zero. See Theorem \ref{Thm-Bound} and Corollary \ref{Non-Vanishing} for the accurate statements. A similar statement holds for the congruence subgroups $\Gamma_0(\mathfrak{n})\subset \textrm{SL}_2(\mathcal{O}_F)$, and by letting $k$ or $N(\mathfrak{n})$ approach $\infty$ this recovers Kumari's result \cite{9} in the case of parallel weight. In Section 4, we prove Selberg's identity on Kloosterman sums over general number fields; this is Theorem \ref{Selberg's identity}. When the narrow class number of $F$ is $1$, Selberg's identity applies to obtain certain additive property of the Fourier coefficients and hence certain non-vanishing relations of Hilbert Poincar\'e series (see Theorem \ref{relation Poincare}). Finally, in order to relate the Hecke operators to Hilbert Poincar\'e series as did by Rankin for the elliptic case in \cite{R}, we consider ad\'elic Hilbert Poincar\'e sereis in the last section and obtain certain vanishing relations.

\section{Hilbert Poincar\'e Series}
In this section, we briefly introduce the basics on Hilbert Poinar\'e series and set up the related notations.

Let $F$ be a totally real number field of degree $n>1$ over $\mathbb{Q}$, with ring of integers $\mathcal{O}=\mathcal{O}_{F}$, different $\mathfrak{d}$, discriminant $D$, distinct real embeddings $\sigma_{1}, \ldots, \sigma_{n}$, and narrow class number $h^+$. For an element $a\in F$, set $a_i=\sigma_i(a)$, the $i$-th component of $a$ under the embedding $F\subset F_\mathbb{R}=F\otimes_\mathbb{Q}\mathbb{R}$. For a subset $S\subset F$, set $S^+$ to be the subset of totally positive elements of $S$, that is, 
\[S^+=\{a\in S: a_i>0, i=1,\ldots, n\}.\]
In particular, $\mathcal{O}^{\times +}$ is the set of totally positive units of $\mathcal{O}$. For ease of notation, from an abelian group or a quotient space of an abelian group $S$, we remove $0$ or the orbit of $0$ and have $S^*$. We shall reserve $A=A(F)>1$ for a constant such that each class of $F^\times/\mathcal{O}^{\times+}$ contains a representative $\mu$ such that $A^{-1}N(\mu)^{\frac{1}{n}}\leq |\mu_i|\leq AN(\mu)^{\frac{1}{n}}$ for each $i$.

Let $\mathfrak{c}$ be a fractional ideal and $\mathfrak{n}$ be a nonzero integral ideal of $F$, and $\Gamma_0(\mathfrak{c}, \mathfrak{n})$ be the congruence subgroup of level $\mathfrak{n}$ defined by
$$\Gamma_0(\mathfrak{c}, \mathfrak{n})=\left\{\begin{pmatrix}
a&b\\c&d
\end{pmatrix}: a,d\in\mathcal{O}, b\in (\mathfrak{cd})^{-1}, c\in\mathfrak{cnd}, ad-bc\in\mathcal{O}^{\times +}\right\}.$$ 

Throughout this section, denote $\Gamma=\Gamma_0(\mathfrak{c}, \mathfrak{n})$.
It is well-known that $\Gamma$ acts discontinuously on $\mathbb{H}^{n}$  with finite covolume via the componentwise M\"obius transformation. For a positive integer $k$, denote by $M_{k}\left(\Gamma\right)$ the space of Hilbert modular forms of parallel weight $(k, \ldots, k)$, level $\Gamma$ and with trivial character, i.e. the space of holomorphic functions $f(z)$ on $\mathbb{H}^{n}$ such that for $\gamma=\begin{pmatrix}
a&b\\c&d
\end{pmatrix} \in \Gamma, f(\gamma(z))=N(c z+d)^{k} f(z)$. Here for $z=\left(z_{1}, \ldots, z_{n}\right) \in \mathbb{H}^{n}$,
$$N(c z+d)^{k}=\prod_{i=1}^{n}\left(c_i z_{i}+d_i\right)^{k}.$$
According to Koecher's principle \cite{F}, each $f$ in $M_{k}\left(\Gamma\right)$ has the following Fourier expansion at $\infty$
$$f(z)=\sum_{\nu \in \mathfrak{c}^{+}\cup\{0\}} a(\nu) \exp (2 \pi i \operatorname{Tr}(\nu z)),$$
where $\operatorname{Tr}(\nu z)=\sum_{i=1}^{n} \nu_i z_{i} $. Since $f$ is invariant under $\begin{pmatrix}
\varepsilon&0\\0&1
\end{pmatrix}$ for $\varepsilon\in\mathcal{O}^{\times +}$, we have $a\left(\varepsilon \nu\right)=a(\nu)$. The Fourier expansion of $f$ at other cusps can be defined and if the constant term $a(0)=0$ at each cusp, $f$ is called a cusp form and the corresponding space is denoted by $S_k(\Gamma)$. The Petersson inner product on $S_{k}(\Gamma)$ is defined by
$$\left\langle f, h\right\rangle=\int_{\Gamma \backslash \mathbb{H}^{n}} f(z) \overline{h}(z)N(y)^{k}d\mu(z),\quad d\mu(z)=\prod_i\frac{d x_{i} d y_{i}}{y_i^{2}}.$$

For each $\mu\in\mathfrak{c}^+$, the Hilbert $\mu$-th Poincar\'e series for $\Gamma$ is defined as
$$P_{\mu}(z ; k, \mathfrak{c},\mathfrak{n}) :=\sum_{\gamma \in Z(\Gamma)N(\Gamma) \backslash \Gamma} N(cz+d)^{-k} e(\mu \gamma z),\qquad \gamma=\begin{pmatrix}
a&b\\c&d
\end{pmatrix},$$ where
\[Z(\Gamma)=\left\{\begin{pmatrix}
\varepsilon&0\\0&\varepsilon
\end{pmatrix}: \varepsilon \in \mathcal{O}^\times\right\},\quad N(\Gamma)=\left\{\begin{pmatrix}
1&b\\0&1
\end{pmatrix}: b \in (\mathfrak{cd})^{-1}\right\}.\] 
The Hilbert Poincar\'e series converges absolutely for $k > 2$ and belongs to $S_k(\Gamma)$, and we shall assume throughout that $k>2$ and is even. 

For any $\varepsilon\in\mathcal{O}^{\times+}$ and $\mu\in\mathfrak{c}^+$, since $\begin{pmatrix}
\varepsilon&0\\0&1
\end{pmatrix}$ normalizes $Z(\Gamma)N(\Gamma)$, it is easy to see that
$P_{\varepsilon\mu}(z ; k, \mathfrak{c},\mathfrak{n})=P_{\mu}(z ; k, \mathfrak{c},\mathfrak{n})$. Therefore, we may consider $\mu\in\mathfrak{c}^+/\mathcal{O}^{\times+}$, and hence we can choose representatives $\mu$ such that $A^{-1}N(\mu)^{\frac{1}{n}}\leq \mu_i\leq AN(\mu)^{\frac{1}{n}}$ for each $i$.

For $f(z)=\sum_{\nu \in \mathfrak{c}^{+}} a(\nu) e(\nu z)\in S_{k}(\Gamma)$, by the definition and the unfolding trick, we have
$$\left\langle f, P_{\mu}(\cdot ; k,\mathfrak{c},\mathfrak{n})\right\rangle=a(\mu)N(\mathfrak{cd})^{-1} N(\mu)^{1-k} \sqrt{D}\left((4 \pi)^{1-k} \Gamma(k-1)\right)^{n}.$$
Consequently, the Hilbert Poincar\'e series span $S_{k}(\Gamma)$. Moreover, by taking the Petersson norm with itself, we see that the
Hilbert Poincar\'e sereis $P_{\mu}(z ; k, \mathfrak{c},\mathfrak{n})$ is identically zero if and only if its $\mu$-th Fourier coefficient vanishes.

It is easy to see that
$$P_{\mu}(z ; k, \mathfrak{c},\mathfrak{n})=\sum_{(c, d)}\sum_{\varepsilon\in\mathcal{O}^{\times +}}\quad N(c z+d)^{-k} e\left(\varepsilon\mu \gamma z\right),$$
where $(c,d)$ runs through all non-associated coprime pairs in $\mathfrak{cnd}\times\mathcal{O}$, and $\gamma \in \Gamma$ is any element with the second row equal to the pair $(c,d)$. Here that $c,d$ are coprime if $c(\mathfrak{cd})^{-1}+d\mathcal{O}=\mathcal{O}$, and recall that pairs $\left(c_{1}, d_{1}\right)$ and $\left(c_{2}, d_{2}\right)$ are associated if there exists a unit $\varepsilon \in \mathcal{O}^\times$ such that $\left(c_{1}, d_{1}\right)=\varepsilon\left(c_{2}, d_{2}\right)$. From this and by the Poisson summation formula, for $\nu\in\mathfrak{c}^+$, the $\nu$-th coefficient $c_k(\nu,\mu)=c_k(\nu,\mu;\mathfrak{c},\mathfrak{n})$ of $P_{\mu}(z;k, \mathfrak{c},\mathfrak{n})$ equals to
\begin{align*}
c_k(\nu,\mu)&=\chi_{\mu}(\nu)+ \left(\frac{N(\nu)}{N(\mu)}\right)^{\frac{k-1}{2}}\frac{(2\pi)^n(-1)^{nk/2}N(\mathfrak{cd})}{D^{1 / 2}}\\
&\qquad \sum_{\varepsilon\in \mathcal{O}^{\times +}}\sum_{c\in(\mathfrak{cnd}/\mathcal{O}^{\times})^{\ast}}\frac{ S_{c(\mathfrak{cd})^{-1}}(\nu, \varepsilon\mu  ; c)}{|N(c)|}NJ_{k-1} \left(\frac{4 \pi \sqrt{\nu\varepsilon\mu}}{|c|}\right).
\end{align*}
Here the representatives $c$ of $(\mathfrak{cnd}/\mathcal{O}^{\times})^{\ast}$ can be chosen such that $A^{-1} |N(c)|^{\frac{1}{n}}\leq |c_i|\leq A |N(c)|^{\frac{1}{n}}$, $i=1,\cdots ,n$, $\chi_{\mu}$ is the characteristic function of the set $\left\{\mu \varepsilon : \varepsilon \in \mathcal{O}^{\times+}\right\}$ and \[NJ_{k-1}\left(\frac{4 \pi \sqrt{\nu\varepsilon\mu}}{|c|}\right)=\prod_{i=1}^{n} J_{k-1} \left(\frac{4 \pi \sqrt{\nu_i\varepsilon_i\mu_i}}{|c_i|}\right).\] Moreover, for any nonzero integral ideal $\mathfrak{m}$, $c\in F^\times$ and $\nu,\mu\in c(\mathfrak{md})^{-1}$, the Kloosterman sum is defined as
$$S_{\mathfrak{m}}(\nu, \mu  ; c) :=\sum_{x\in(\mathcal{O}/\mathfrak{m})^{\times}}e\left(\frac{\nu x+\mu x^{-1}}{c}\right),\quad xx^{-1} \equiv 1 \bmod \mathfrak{m}.$$ Here $e(\nu)=\exp(2\pi i\textrm{Tr}(\nu))$.  When $c\in\mathcal{O}$ and $\mathfrak{m}=(c)$, we shall simply write $S(\nu,\mu;c)=S_{(c)}(\nu,\mu;c)$.

\section{Non-Vanishing Condition}

In this section, we estimate the Fourier coefficient 
$$c_k(\mu,\mu)= 1+\frac{(2\pi)^n(-1)^{nk/2}N(\mathfrak{cd})}{D^{1 / 2}}\sum_{\varepsilon\in \mathcal{O}^{\times +}}\sum_{c\in(\mathfrak{cnd}/\mathcal{O}^{\times})^{\ast}}\frac{ S_{c(\mathfrak{cd})^{-1}}(\nu, \varepsilon\mu  ; c)}{|N(c)|}NJ_{k-1} \left(\frac{4 \pi \sqrt{\varepsilon\mu^2}}{|c|}\right)$$
to get the non-vanishing condition on Hilbert Poincar\'e series. Since $P_{\mu}(z ; k, \mathfrak{c},\mathfrak{n})=P_{\varepsilon\mu}(z ; k, \mathfrak{c},\mathfrak{n})$ for any $\varepsilon\in\mathcal{O}^{\times +}$, we shall choose $\mu$ such that $A^{-1}|N(\mu)|^{\frac{1}{n}}\leq |\mu_i|\leq A |N(\mu)|^{\frac{1}{n}}$, $i=1,\cdots ,n$.

We first prove the following upper bound for $S_{\mathfrak{m}}(\nu, \mu ; c)$, which is slightly more general than Theorem 10 of \cite{11}. Let $\mathfrak{m}$ be any nonzero integral ideal, $c\in F^\times$ and $\nu,\mu\in c(\mathfrak{md})^{-1}$. Denote by $pr(\mathfrak{m})$ the number of distinct prime ideals that divide $\mathfrak{m}$. Let $f_\mathfrak{p}=f(\mathfrak{p}/p)$ be the inertial degree at the prime ideal $\mathfrak{p}$ over $p$ and set $f=\sum_{\mathfrak{p}\mid 2}f_\mathfrak{p}$. Moreover, let $v_\mathfrak{p}$ be the valuation at $\mathfrak{p}$ and set
\[N_{\nu, \mu}(\mathfrak{m})=\prod_{\mathfrak{p}\mid \mathfrak{m}} N(\mathfrak{p})^{\min \left\{v_{\mathfrak{p}}(\nu), v_{\mathfrak{p}}\left(\mu\right), v_{\mathfrak{p}}(\mathfrak{m})-v_{\mathfrak{p}}(\mathfrak{d})\right\}}.\]

\begin{Prop}\label{Kloosterman sum ideal}
We have
$$\left| S_{\mathfrak{m}}(\nu, \mu ; c) \right| \leq 2^{n+f / 2} \sqrt{\left|D\right|} \sqrt{N_{\nu, \mu}(\mathfrak{m})} 2^{pr(\mathfrak{m})} \sqrt{N(\mathfrak{m})}.$$
\end{Prop} 
\begin{proof} 
Define two characters $\psi,\phi: (\mathcal{O}/\mathfrak{m})^{\times} \rightarrow \mathbb{C}^{\times}$ by 
$$\psi(h)=e(\nu h/c),\quad \phi(h)=e(\mu h/c).$$
By Proposition 9 of \cite{11}, we have 
$$\left| \sum_{h\in(\mathcal{O}/\mathfrak{m})^{\times}}e\left(\frac{\nu h + \mu h^{-1}}{c}\right) \right|=\left| \prod_{\mathfrak{p}} S\left(\psi_\mathfrak{p},\phi_\mathfrak{p},\mathfrak{p}^{v_{\mathfrak{p}}(\mathfrak{m})}\right) \right|
\leq \prod_{\mathfrak{p}\mid\mathfrak{m}} c_{\mathfrak{p}} N(\mathfrak{p})^{v_{\mathfrak{p}}(\mathfrak{m})-\frac{N_{\mathfrak{p}}}{2}},$$
where $N_{\mathfrak{p}}=\max \left\{0,-v_{\mathfrak{p}}(\nu)-v_{\mathfrak{p}(\mathfrak{d})}+v_{\mathfrak{p}}(\mathfrak{m}),-v_{\mathfrak{p}}\left( \mu\right)-v_{\mathfrak{p}(\mathfrak{d})}+v_{\mathfrak{p}}(\mathfrak{m})\right\}$, $c_{\mathfrak{p}}=2$ if $2\nmid N(\mathfrak{p})$ and $c_{\mathfrak{p}}=2N(\mathfrak{p})^{v+\frac{1}{2}}$ if $v=v_{\mathfrak{p}}(2)\geq 1$. Then the upper bound follows easily.
\end{proof}

\begin{Thm}\label{Thm-Bound}
Let $\mathfrak{c}$ be an integral ideal, and $\mu\in\mathfrak{c}^+$. For any $0 <\eta<1$, if 
$$|N(\mu)|< C\left((k-1)^{n-nk}N(\mathfrak{cn})^{-k+1+\eta} \right)^{-1/(k-\frac{1}{2})},$$
then $P_{\mu}(z ; k, \mathfrak{c},\mathfrak{n})$ is not identically zero. Here $C$ is an effective positive constant depending on $F$ and $\eta$ only.
\end{Thm}
\begin{proof} 
It is well-known that $|J_{\nu}(x)|\leq 1$ and $|J_{\nu}(x)|\leq\frac{\left(\frac{1}{2}x\right)^{\nu}}{\Gamma(\nu+1)}$ for $\nu \geq 0$ and  $x > 0$ (See (10.14.1) and (10.14.4) of \cite{O}). When $\nu=k-1$, by Stirling's lower bound, 
\[|J_{k-1}(x)|\leq \frac{1}{\sqrt{2(k-1)\pi}}\left(\frac{ex}{2(k-1)}\right)^{k-1}\leq \left(\frac{ex}{2(k-1)}\right)^{k-1}.\]
It follows that for any $0\leq \eta<1$, 
\[|J_{k-1}(x)|\leq \min\left\{1, \left(\frac{ex}{2(k-1)}\right)^{k-1}\right\}\leq \left(\frac{ex}{2(k-1)}\right)^{k-1-\eta}.\]
Now choosing $\eta=0$ for the factors of $NJ_{k-1}$ with $\varepsilon_i\leq 1$ and $\eta\in (0,1)$ for the factors with $\varepsilon_i>1$, from the choices of $c$ and $\mu$,  above bounds give that 
$$ NJ_{k-1} \left(\frac{4 \pi \sqrt{\varepsilon\mu^2}}{|c|}\right) \leq C_1(\frac{4\pi e}{2k-2})^{nk-n}N(\mu)^{k-1}
|N(c)|^{-k+1+\eta} \prod_{\left|\varepsilon_j\right|>1}\left|\varepsilon_j\right|^{-\eta}$$
where $C_1$ can be taken as $C_1=A^n$. Note that for any $c\in \mathfrak{cnd}$, $c(\mathfrak{cd})^{-1}$ is an integral ideal. Since $2^{pr(\mathfrak{m})} \leq C_2\sqrt{N(\mathfrak{m})}$ and $N_{\mu,\nu}(\mathfrak{m})\leq N(\mu)$ for any integral ideal $\mathfrak{m}$, we have
\begin{align*}
\left| S_{c(\mathfrak{cd})^{-1}}(\mu, \varepsilon \mu  ; c) \right| 
&\leq  2^{n+\frac{f}{2}} 2^{pr(c(\mathfrak{cd})^{-1})} \sqrt{\left|D\right|} \sqrt{N_{\mu, \varepsilon \mu }(c(\mathfrak{cd})^{-1})}\sqrt{N(c(\mathfrak{cd})^{-1})}\\
&\leq  C_3\frac{\sqrt{N(\mu)}|N(c)|}{N(\mathfrak{cd})}
\end{align*}
by Proposition \ref{Kloosterman sum ideal}. Now we have
\begin{align*}
&|c(\mu;\mu,\mathfrak{c},\mathfrak{n})-1|\\
=&\left|\frac{(2\pi)^n(-1)^{nk/2}N(\mathfrak{cd})}{D^{1 / 2}}\sum_{\varepsilon\in \mathcal{O}^{\times+}}\sum_{c\in \left(\mathfrak{cnd}/\mathcal{O}^\times\right)^*}  S_{c(\mathfrak{cd})^{-1}}(\mu, \varepsilon\mu  ; c) |N(c)|^{-1} J_{k-1} \left(\frac{4 \pi \sqrt{\varepsilon\mu^2}}{|c|}\right)\right|\\
\leq & C_4\sum_{\varepsilon\in \mathcal{O}^{\times+}}\sum_{c\in \left(\mathfrak{cnd}/\mathcal{O}^\times\right)^*} 
\sqrt{N(\mu)}
J_{k-1} \left(\frac{4 \pi \sqrt{\varepsilon\mu^2}}{|c|}\right)\\
\leq & C_5(\frac{2\pi e}{k-1})^{nk -n} N(\mu)^{k-\frac{1}{2}}
\sum_{\varepsilon\in \mathcal{O}^{\times+}}\prod_{\left|\varepsilon_j\right|>1}\left|\varepsilon_j\right|^{-\eta}
\sum_{c\in(\mathfrak{cnd}-\{0\})/\mathcal{O}^{\times}} |N(c)|^{1-k+\eta}\\
\leq & C_6(\frac{2\pi e}{k-1})^{nk -n} N(\mu)^{k-\frac{1}{2}}
\sum_{\varepsilon\in \mathcal{O}^{\times+}}\prod_{\left|\varepsilon_j\right|>1}\left|\varepsilon_j\right|^{-\eta}
\frac{\zeta_F(k-1-\eta)}{N(\mathfrak{cnd})^{k-1-\eta}}\\
\leq & C_7(\frac{2\pi e}{k-1})^{nk -n} N(\mu)^{k-\frac{1}{2}}
\sum_{\varepsilon\in \mathcal{O}^{\times+}}\prod_{\left|\varepsilon_j\right|>1}\left|\varepsilon_j\right|^{-\eta}
N(\mathfrak{cnd})^{-k+1+\eta}.
\end{align*}
We emphasize that all of the constants involved above are effective and depend on $F$ only. By Dirichlet's unit theorem, 
\[\sum_{\varepsilon\in \mathcal{O}^{\times+}}\prod_{\left|\varepsilon_j\right|>1}\left|\varepsilon_j\right|^{-\eta}\leq C_8,\]
where $C_8$ is a positive constant that depends on $F$ and $\eta$ only (see \cite{L}). Therefore, 
\[|c(\mu;\mu,\mathfrak{c},\mathfrak{n})-1|\leq C_9(\frac{2\pi e}{k-1})^{nk -n} N(\mu)^{k-\frac{1}{2}}N(\mathfrak{cnd})^{-k+1+\eta},\]
with $C_9$ effective and depending on $F$ and $\eta$ only.

It follows that there exists an effective constant $C$ depending on $F$ and $\eta$ only, such that if 
$$|N(\mu)|< C\left((k-1)^{n-nk} N(\mathfrak{cn})^{-k+1+\eta} \right)^{-1/(k-\frac{1}{2})},$$
then $|c_k(\mu,\mu)-1|<1$. Hence $c_k(\mu,\mu)\neq 0$  and $P_{\mu}(z ; k, \mathfrak{c},\mathfrak{n})$ is not identically zero. 
\end{proof}

By choosing an integral representative in the narrow ideal class of a fractional ideal, we obtain the following non-vanishing condition when $\mathfrak{c}$ is a fractional ideal.

\begin{Cor}\label{Non-Vanishing}
Let $\mathfrak{c}$ be a fractional ideal and $\alpha\in F^+$ with $\alpha\mathfrak{c}$ being integral. If $\mu\in\mathfrak{c}^+$ satisfies
$$|N(\mu)| < C(k-1)^{\frac{2nk-2n}{2k-1}} N(\mathfrak{cn})^{\frac{2k-3}{2k-1}} N(\alpha)^{-\frac{2}{2k-1}},$$ 
then $P_{\mu}(z ; k, \mathfrak{c},\mathfrak{n})$ is not identically zero. Here $C$ is an effective positive constant depending on $F$ only.
\end{Cor}
\begin{proof}
It is clear that
$$
P_{\alpha\mu}(z ; k, \alpha \mathfrak{c}, \mathfrak{n})=P_{\mu}(\alpha z ; k, \mathfrak{c}, \mathfrak{n})
$$
and $c_k(\alpha\nu;\alpha\mu,\alpha\mathfrak{c},\mathfrak{n})=c_k(\nu;\mu,\mathfrak{c},\mathfrak{n})$. By the preceding theorem, if
\[|N(\alpha\mu)| < C(k-1)^{\frac{2nk-2n}{2k-1}} N(\alpha\mathfrak{cn})^{\frac{2k-2-2\eta}{2k-1}},\]
or 
\[|N(\mu)| < C(k-1)^{\frac{2nk-2n}{2k-1}} N(\mathfrak{cn})^{\frac{2k-2-2\eta}{2k-1}}N(\alpha)^{-\frac{1+2\eta}{2k-1}},\]
$P_{\alpha\mu}(z ; k, \alpha \mathfrak{c}, \mathfrak{n})$ and hence $P_{\mu}(z ; k, \mathfrak{c}, \mathfrak{n})$ are not identically zero. In particular, set $\eta=\frac{1}{2}$ and we finish the proof.
\end{proof}

Finally, we state a similar result for Hilbert Poincar\'e series on the congruence subgroup
$$\Gamma_{0}(\mathfrak{n})=\left\{\left(\begin{array}{ll}a & b \\ c & d\end{array}\right) \in \text{SL}_2(\mathcal{O}_F), c \in \mathfrak{n}\right\},$$
and omit the almost identical details. The Hilbert Poincar\'e series is now defined for $\mu\in(\mathfrak{d}^{-1})^+$ by 
$$P_{\mu}'(z ; k, \mathfrak{n})=\sum_{\gamma \in N(\Gamma) \backslash \Gamma} N(c z+d)^{-k} e(\mu \gamma z),\quad \gamma=\left(\begin{array}{ll}a & b \\ c & d\end{array}\right),$$ where $N(\Gamma)=\left\{\left(\begin{array}{ll}1 & b \\ 0 & 1\end{array}\right): b \in \mathcal{O}\right\}$. 

\begin{Thm}\label{Thm-SL}
For $\mu\in (\mathfrak{d}^{-1})^+$, if 
$$|N(\mu)| \leq C(k-1)^{n-\frac{3n}{2k-1}}N(\mathfrak{n})^{\frac{k-3/4}{k-1/2}},$$
then $P_{\mu}'(z ; k, \mathfrak{n})$ is not identically zero. Here $C$ is an effective positive constant depending on $F$ only. 
\end{Thm}

\begin{Rmk}
(1) Our treatment actually works for $n=1$ as well, but our result is weaker than Rankin's \cite{R} in this case. Indeed, for any $\varepsilon>0$, $k^{-\varepsilon} \leq exp(-\frac{\text{log } k}{\text{log log }k})$ for sufficiently large $k$, so the upper bound in \cite{R} is roughly $k^2$, while our result gives roughly $k$.

(2) Such upper bound tends to infinity when either $N(\mathfrak{n})$ or $k$ tends to infinity. Therefore, Theorem \ref{Thm-SL} recovers Kumari's Theorem 4.1 and Theorem 4.3 in the case of parallel weight.
\end{Rmk}

\section{Selberg's Identity}

In this section, we assume that the different $\mathfrak{d}=(\delta)$ is principal. We shall prove Selberg's identity in this case and then apply it to the non-vanishing problem. For $\nu,\mu\in \mathfrak{d}^{-1}$ and $q\in \mathcal{O}$, in this section we focus on the Kloosterman sum
\begin{equation*}
S(\nu, \mu ; q)={\sum_{x \bmod q}}^{\prime} e\left(\frac{\nu x+\mu x^{-1}}{q}\right).
\end{equation*}
Here $\sum'_{x\bmod q}$ is to sum over invertible elements of $\mathcal{O}/(q)$.

\begin{Lem}
Let $p$ be any prime element in $\mathcal{O}$, $\varepsilon_1,\varepsilon_2\in\mathcal{O}^\times$ and integer $e \geq 1$. Then for any $r\in\mathcal{O}$ with $p|r$, we have
\begin{equation*}\label{vanishing Kloosterman}
S\left(\delta^{-1}\varepsilon_1, \delta^{-1}r ; \varepsilon_2 p^{e}\right)=\begin{cases}
-1,& \text{if}\quad e=1,\\
0, & \text{if}\quad e > 1.
\end{cases}\end{equation*}
\end{Lem}
\begin{proof}
By a change of variable,
$$S(\delta^{-1}\varepsilon_1, \delta^{-1}r ; \varepsilon_2 p^{e})={\sum_{x\bmod p^e}}'  e\left(\frac{\varepsilon_1 x+r x^{-1}}{\delta \varepsilon_2 p^{e}}\right)
={\sum_{x\bmod p^e}}^{\prime}  e\left(\frac{x+\varepsilon_2^{-2}\varepsilon_1 r x^{-1}}{\delta p^{e}}\right),$$
so we only have to prove the statement when $\varepsilon_1=\varepsilon_2=1$.

We note that $x \mapsto x+r x^{-1}$ maps the set $\left(\mathcal{O}/q\right)^\times$ to itself, since $p\mid r$. Now we show that such a map is bijective, for which we only have to show the injectivity. If
$$x+r x^{-1} \equiv y+r y^{-1} \quad\bmod p^e,$$
where $x$ and $y$ are coprime to $p^e$, then
\begin{align*}
x y(x-y)+r(y-x) &\equiv 0 \quad\bmod p^e,\\
(x y-r)(x-y) &\equiv 0 \quad\bmod p^e,
\end{align*}
showing that $x \equiv y\bmod p^e$ as claimed, since $x y-r$ is coprime to $p^e$. 

Such bijection simplifies the Kloosterman sum and gives
\begin{align*}
S(\delta^{-1}, \delta^{-1}r ; p^e) &=\sum_{x\bmod p^e}\nolimits^{\prime} e\left(\frac{x}{\delta p^e}\right)=\sum_{x \bmod p^e} e\left(\frac{x}{\delta p^e}\right)-\sum_{\substack{ x \bmod p^e \\ p \mid x}} e\left(\frac{x}{\delta p^e}\right)\\
&= \sum_{x \bmod p^e} e\left(\frac{x}{\delta p^e}\right)-\sum_{x \bmod p^{e-1} } e\left(\frac{x}{\delta p^{e-1}}\right).
\end{align*}
Note that $\sum_{x \bmod p^e} e\left(\frac{x}{\delta p^e}\right)=0$ if $e\geq 1$ and $=1$ if $e=0$. From this, the statement follows.
\end{proof}

Now we can prove the following version of Selberg's identity.

\begin{Thm}
Assume that $\mathfrak{d}=(\delta)$ is principal and $q\in\mathcal{O}$ is a product of prime elements. Then for any $\nu,\mu\in\mathcal{O}$, we have
\begin{equation*}\label{Selberg's identity}
S(\delta^{-1}\nu, \delta^{-1}\mu ; q)=\sum_{(d) \mid(\nu, \mu, q)} N((d)) S\left(\delta^{-1}, \delta^{-1}\frac{\nu\mu}{d^{2}} ; \frac{q}{d}\right),
\end{equation*}
where the summation is over all principal ideals $(d)$ that divide $(\nu,\mu,q)$. 
\end{Thm} 

By a simple change of variable, we see that choosing different generators of the principal ideal $(d)$ does not affect the Kloosterman sums on the right-hand side. 

\begin{proof}
It is easy to see that the identity is multiplicative in $q$; namely, if the identity holds for $q_1$ and $q_2$ with $(q_1,q_2)=\mathcal{O}$, then it also holds for $q_1q_2$ (see \cite{15} for the proof on  classical Kloosterman sums). Therefore, we only have to prove the identity for $q=p^e$, with $p$ a prime element and $e\geq 0$. We proceed by induction on $e$. The identity is trivial for $e=0$, so we assume that $e\geq 1$ and it holds for $q=p^{e-1}$.

If $(\nu, p)=1$ or $(\mu, p)=1$, we get from definition by a simple change of variable that $S(\delta^{-1}\nu,\delta^{-1}\mu ; p^e)=S(\delta^{-1}, \delta^{-1}\nu\mu ; p^e)$. In this case, we only need to show $S(\delta^{-1}, \delta^{-1}\nu\mu ; p^e)=S(\delta^{-1}, \delta^{-1}\frac{\nu\mu}{\varepsilon^2} ; \frac{p^e}{\varepsilon})$ for any unit $\varepsilon$, which follows immediately from a change of variable. So from now on we assume that $p$ divides both $\nu$ and $\mu$. Then by induction, the identity simplifies to
\begin{equation}\label{simplified form}
S(\delta^{-1}\nu,\delta^{-1}\mu ; p^e)=S(\delta^{-1}, \delta^{-1}\nu\mu ; p^e)+N((p)) S\left(\frac{\delta^{-1}\nu}{p}, \frac{\delta^{-1}\mu}{p} ; p^{e-1}\right) .
\end{equation}
If $e=1$, the identity \eqref{simplified form} is valid, because
$$S(\delta^{-1}\nu,\delta^{-1}\mu ; p)=N((p))-1, \quad S(\delta^{-1}, \delta^{-1}\nu\mu ; p)=-1, \quad S\left(\frac{\delta^{-1}\nu}{p}, \frac{\delta^{-1}\mu}{p} ; 1\right)=1.$$
Now we assume that $e \geq 2$. Since $(x, p^e)=1$ if and only if $(x, p^{e-1})=1$, hence in \eqref{simplified form} the left hand side is equal to 
\begin{align*}
{\sum_{x \bmod p^e}}'e\left(\frac{\nu x+\mu x^{-1}}{\delta p^e}\right)&={\sum_{x \bmod p^{e-1}}}'\sum_{y\bmod p}e\left(\frac{\nu (x+yp^{e-1})+\mu (x+yp^{e-1})^{-1}}{\delta p^e}\right).
\end{align*}
Since $(x+yp^{e-1})^{-1}=x^{-2}(x-yp^{e-1})$ and $p\mid (\nu,\mu)$, this is equal to 
\begin{align*}
{\sum_{x \bmod p^{e-1}}}'\sum_{y\bmod p}e\left(\frac{\nu x+\mu x^{-1}}{\delta p^e}\right)&=N((p))S\left(\frac{\nu}{\delta p},\frac{\mu}{\delta p}; p^{e-1}\right),
\end{align*}
which is exactly what we want since $S(\delta^{-1}, \delta^{-1}\nu\mu ; p^e)=0$. This finishes the proof.
\end{proof}

For the rest of this section, we assume furthermore that the narrow class number $h^+$ of $F$ is 1, so that every nonzero element in $\mathcal{O}$ satisfies the assumption in Theorem \ref{Selberg's identity}. 

\begin{Cor} \label{relation Kloosterman} Assume $h^+=1$. For a nonzero element of $q\in\mathcal{O}$, positive integers $m,n$ and $\nu,\mu\in\mathfrak{d}^{-1}$, let $p$ be a prime element dividing $q$ but coprime to $\delta\nu$ and $\delta\mu$. Then
$$S(\nu p^m, \mu p^n ; q) = S(\nu,\mu p^{m+n}; q)+N((p)) S\left(\nu p^{m-1}, \mu p^{n-1}; \frac{q}{p}\right).$$
\end{Cor}
\begin{proof}
Introducing the generator $\delta$ of $\mathfrak{d}$, we prove instead
$$S(\delta^{-1}\nu p^m, \delta^{-1}\mu p^n ; q) = S(\delta^{-1}\nu,\delta^{-1}\mu p^{m+n}; q)+N((p)) S\left(\delta^{-1}\nu p^{m-1},\delta^{-1} \mu p^{n-1}; \frac{q}{p}\right)$$
for $\nu,\mu\in\mathcal{O}$ satisfying
$p\nmid \nu, p \nmid \mu$. By the preceding theorem, we have
\begin{align*}
S(\delta^{-1}\nu p^m, \delta^{-1}\mu p^n; q)&=\sum_{d \mid(\nu p^m, \mu p^n, q)} N((d)) S\left(\delta^{-1}, \delta^{-1}\frac{\nu\mu p^{m+n}}{d^{2}} ; \frac{q}{d}\right)\\
&= \left(\sum_{\substack{ d \mid(\nu p^m, \mu p^n, q)\\ p\nmid d}} +\sum_{\substack{ d \mid(\nu p^m, \mu p^n, q)\\ p\mid d}} \right) N((d)) S\left(\delta^{-1}, \delta^{-1}\frac{\nu\mu p^{m+n}}{d^{2}} ; \frac{q}{d}\right)\\
&= S(\delta^{-1}\nu,\delta^{-1}\mu p^{m+n}; q)+N((p)) S\left(\delta^{-1}\nu p^{m-1}, \delta^{-1}\mu p^{n-1}; \frac{q}{p}\right),
\end{align*}
which is what we want.
\end{proof}

Now we rewrite the series expression for the Fourier coefficient $c_k(\nu,\mu)$ and then introduce and work on a symmetric variant. For any $\varepsilon\in\mathcal{O}^{\times}$, 
\begin{align*}
S_{\mathfrak{m}}(\nu, \mu\varepsilon^2 ; c) &=\sum_{x\in (\mathcal{O}/\mathfrak{m})^{\times}} e\left(\frac{\nu x+\mu\varepsilon^2 x^{-1}}{c}\right)
=\sum_{x\in (\mathcal{O}/\mathfrak{m})^{\times}} e\left(\frac{\varepsilon^{-1}\nu x+\varepsilon\mu x^{-1}}{\varepsilon^{-1}c}\right)\\
&=\sum_{x\in (\mathcal{O}/\mathfrak{m})^{\times}} e\left(\frac{\nu x+\mu x^{-1}}{\varepsilon^{-1}c}\right)=S_{\mathfrak{m}}(\nu, \mu; c/\varepsilon),
\end{align*}
so the $\nu$-th coefficient $c_k(\nu,\mu)$  of $P_{\mu}(z;k, \mathfrak{c},\mathfrak{n})$ equals
\begin{align*}
&\chi_{\mu}(\nu)+ \left(\frac{N(\nu)}{N(\mu)}\right)^{\frac{k-1}{2}}\frac{(2\pi)^n(-1)^{nk/2}N(\mathfrak{cd})}{D^{1 / 2}}\\
&\qquad\times\sum_{\varepsilon\in \mathcal{O}^{\times +}}\sum_{c\in(\mathfrak{cnd}/\mathcal{O}^{\times})^{\ast}}\frac{ S_{c(\mathfrak{cd})^{-1}}(\nu, \varepsilon\mu  ; c)}{|N(c)|}NJ_{k-1} \left(\frac{4 \pi \sqrt{\nu\varepsilon\mu}}{|c|}\right)\\
=&\chi_{\mu}(\nu)+ \left(\frac{N(\nu)}{N(\mu)}\right)^{\frac{k-1}{2}}\frac{(2\pi)^n(-1)^{nk/2}N(\mathfrak{cd})}{D^{1 / 2}}\\
&\qquad\times \sum_{\varepsilon\in \mathcal{O}^{\times +}/(\mathcal{O}^{\times})^2}\sum_{\eta\in\mathcal{O}^{\times}}
\sum_{c\in(\mathfrak{cnd}/\mathcal{O}^{\times})^{\ast}}
\frac{ S_{c(\mathfrak{cd})^{-1}}(\nu, \varepsilon\eta^2\mu  ; c)}{|N(c)|}NJ_{k-1} \left(\frac{4 \pi \sqrt{\nu\varepsilon\eta^2\mu}}{|c|}\right)\\
=&\chi_{\mu}(\nu)+ \left(\frac{N(\nu)}{N(\mu)}\right)^{\frac{k-1}{2}}\frac{(2\pi)^n(-1)^{nk/2}N(\mathfrak{cd})}{D^{1 / 2}}\\
&\qquad \times\sum_{\varepsilon\in \mathcal{O}^{\times +}/(\mathcal{O}^{\times})^2}\sum_{\eta\in\mathcal{O}^{\times}}
\sum_{c\in(\mathfrak{cnd}/\mathcal{O}^{\times})^{\ast}}
\frac{ S_{c(\mathfrak{cd})^{-1}}(\nu, \varepsilon\mu  ; c\eta)}{|N(c)|}NJ_{k-1} \left(\frac{4 \pi \sqrt{\nu\varepsilon\mu}}{|c\eta|}\right)\\
=&\chi_{\mu}(\nu)+ \left(\frac{N(\nu)}{N(\mu)}\right)^{\frac{k-1}{2}}\frac{(2\pi)^n(-1)^{nk/2}N(\mathfrak{cd})}{D^{1 / 2}}\\
&\qquad\times\sum_{\varepsilon\in \mathcal{O}^{\times +}/(\mathcal{O}^{\times})^2}
\sum_{c\in(\mathfrak{cnd})^{\ast}}
\frac{ S_{c(\mathfrak{cd})^{-1}}(\nu, \varepsilon\mu  ; c)}{|N(c)|}NJ_{k-1} \left(\frac{4 \pi \sqrt{\nu\varepsilon\mu}}{|c|}\right).
\end{align*}

Let $q$ be a totally positive generator of $\mathfrak{c}$, and $\nu,\mu\in\mathfrak{c}^+, c\in\mathfrak{cnd}$, we have $\frac{\nu}{q\delta},\frac{\mu}{q\delta}\in \mathfrak{d}^{-1}$, $\frac{c}{q\delta}\in\mathcal{O}$, hence
$S_{c(\mathfrak{ad})^{-1}}(\nu,\mu;c)=S(\frac{\nu}{q\delta},\frac{\mu}{q\delta};\frac{c}{q\delta})$. Now the $\nu$-th coefficient $c_k(\nu,\mu)$  of $P_{\mu}(z;k, \mathfrak{c},\mathfrak{n})$ equals to
$$ \chi_{\mu}(\nu)+  \left(\frac{N(\nu)}{N(\mu)}\right)^{\frac{k-1}{2}} \frac{(2\pi)^n(-1)^{nk/2}}{D^{1 / 2}}
\sum_{\varepsilon\in \mathcal{O}^{\times +}/(\mathcal{O}^{\times})^2}\sum_{c\in\mathfrak{n}^{\ast}}  S\left(\frac{\nu}{q\delta},\frac{\varepsilon\mu}{q\delta};c\right)  |N(c)|^{-1} NJ_{k-1} \left(\frac{4 \pi \sqrt{\nu\varepsilon\mu}}{|q\delta c|}\right).$$

In the following, we introduce and work on $\widetilde{c_k}(\nu, \mu)=N(\mu)^{k-1}c_k(\nu, \mu)$. Explicitly,
\begin{align*}
N(\nu\mu)^{-\frac{k-1}{2}}\widetilde{c_k}(\nu, \mu)&=\chi_{\mu}(\nu)+  \frac{(2\pi)^n(-1)^{nk/2}}{D^{1 / 2}}\\
&\qquad\times
\sum_{\varepsilon\in \mathcal{O}^{\times +}/(\mathcal{O}^{\times})^2}\sum_{c\in\mathfrak{n}^{\ast}}  S\left(\frac{\nu}{q\delta},\frac{\varepsilon\mu}{q\delta};c\right)  |N(c)|^{-1} NJ_{k-1} \left(\frac{4 \pi \sqrt{\nu\varepsilon\mu}}{|q\delta c|}\right).
\end{align*}
By a change of variable in the Kloosterman sum, we see that $\widetilde{c_k}(\nu, \mu)$ is symmetric in $\mu$ and $\nu$.

\begin{Cor} \label{relation coefficient} Assume $h^+=1$. For $\nu,\mu\in\mathfrak{c}^+=(q)^+$, positive integers $m,n$, and a totally positive prime element $p$ coprime to $\nu\mu q^{-2}\mathfrak{n}$, we have
$$\widetilde{c_{k}}(\nu p^{m}, \mu p^{n} )=\widetilde{c_{k}}(\nu, \mu p^{m+n})+N((p))^{k-1} \widetilde{c_{k}}(\nu p^{m-1}, \mu p^{n-1} ).$$
\end{Cor}
\begin{proof}
The $\chi$-terms of both sides clearly match and we only have to prove 
\[\widetilde{c_{k}}'(\nu p^{m}, \mu p^{n} )=\widetilde{c_{k}}'(\nu, \mu p^{m+n})+\widetilde{c_{k}}'(\nu p^{m-1}, \mu p^{n-1} )\]
for 
\[\widetilde{c_k}'(\nu, \mu):=\sum_{\varepsilon\in \mathcal{O}^{\times +}/(\mathcal{O}^{\times})^2}\sum_{c\in\mathfrak{n}^{\ast}}  S\left(\frac{\nu}{q\delta},\frac{\varepsilon\mu}{q\delta};c\right)  |N(c)|^{-1} NJ_{k-1} \left(\frac{4 \pi \sqrt{\nu\varepsilon\mu}}{|q\delta c|}\right).\]
We have
\begin{align*}
&\widetilde{c_{k}}'(\nu, \mu p^{m+n})+\widetilde{c_{k}}'(\nu p^{m-1}, \mu p^{n-1} )\\
=&\sum_{\varepsilon\in \mathcal{O}^{\times +}/(\mathcal{O}^{\times})^2}\sum_{c\in\mathfrak{n}^{\ast}-p\mathfrak{n}^{\ast}}  S\left(\frac{\nu}{q\delta},\frac{\varepsilon\mu p^{m+n}}{q\delta};c\right)  |N(c)|^{-1} NJ_{k-1} \left(\frac{4 \pi \sqrt{\nu\varepsilon\mu p^{m+n}}}{|q\delta c|}\right)\\
&+\sum_{\varepsilon\in \mathcal{O}^{\times +}/(\mathcal{O}^{\times})^2}\sum_{c\in p\mathfrak{n}^{\ast}}  S\left(\frac{\nu}{q\delta},\frac{\varepsilon\mu p^{m+n}}{q\delta};c\right)  |N(c)|^{-1} NJ_{k-1} \left(\frac{4 \pi \sqrt{\nu\varepsilon\mu p^{m+n}}}{|q\delta c|}\right)\\
&+\sum_{\varepsilon\in \mathcal{O}^{\times +}/(\mathcal{O}^{\times})^2}\sum_{c\in\mathfrak{n}^{\ast}}  S\left(\frac{\nu p^{m-1}}{q\delta},\frac{\varepsilon\mu p^{n-1}}{q\delta};c\right)  |N(c)|^{-1} NJ_{k-1} \left(\frac{4 \pi \sqrt{\nu\varepsilon\mu p^{m+n}}}{|q\delta c p|}\right)\\
=&\sum_{\varepsilon\in \mathcal{O}^{\times +}/(\mathcal{O}^{\times})^2}\sum_{c\in\mathfrak{n}^{\ast}-p\mathfrak{n}^{\ast}}  S\left(\frac{\nu}{q\delta},\frac{\varepsilon\mu p^{m+n}}{q\delta};c\right)  |N(c)|^{-1} NJ_{k-1} \left(\frac{4 \pi \sqrt{\nu\varepsilon\mu p^{m+n}}}{|q\delta c|}\right)\\
&+\sum_{\varepsilon\in \mathcal{O}^{\times +}/(\mathcal{O}^{\times})^2}\sum_{c\in \mathfrak{n}^{\ast}}  S\left(\frac{\nu}{q\delta},\frac{\varepsilon\mu p^{m+n}}{q\delta},pc\right)  |N(pc)|^{-1} NJ_{k-1} \left(\frac{4 \pi \sqrt{\nu\varepsilon\mu p^{m+n}}}{|q\delta p c|}\right) \\
&+\sum_{\varepsilon\in \mathcal{O}^{\times +}/(\mathcal{O}^{\times})^2}\sum_{c\in\mathfrak{n}^{\ast}}  S\left(\frac{\nu p^{m-1}}{q\delta},\frac{\varepsilon\mu p^{n-1}}{q\delta};c\right)  |N(c)|^{-1} NJ_{k-1} \left(\frac{4 \pi \sqrt{\nu\varepsilon\mu p^{m+n}}}{|q\delta c p|}\right),
\end{align*}
which, by Corollary \ref{relation Kloosterman}, is equal to
\begin{align*}
&\sum_{\varepsilon\in \mathcal{O}^{\times +}/(\mathcal{O}^{\times})^2}\sum_{c\in\mathfrak{n}^{\ast}-p\mathfrak{n}^{\ast}}  S\left(\frac{\nu}{q\delta},\frac{\varepsilon\mu p^{m+n}}{q\delta};c\right)  |N(c)|^{-1} NJ_{k-1} \left(\frac{4 \pi \sqrt{\nu\varepsilon\mu p^{m+n}}}{|q\delta c|}\right)\\
&\sum_{\varepsilon\in \mathcal{O}^{\times +}/(\mathcal{O}^{\times})^2}\sum_{c\in\mathfrak{n}^{\ast}}  S\left(\frac{\nu p^{m}}{q\delta},\frac{\varepsilon\mu p^{n}}{q\delta};c p\right)  |N(c p)|^{-1} NJ_{k-1} \left(\frac{4 \pi \sqrt{\nu\varepsilon\mu p^{m+n}}}{|q\delta c p|}\right)\\
=&\sum_{\varepsilon\in \mathcal{O}^{\times +}/(\mathcal{O}^{\times})^2}\sum_{c\in\mathfrak{n}^{\ast}-p\mathfrak{n}^{\ast}}   S\left(\frac{\nu p^{m}}{q\delta},\frac{\varepsilon\mu p^{n}}{q\delta};c\right)  |N(c)|^{-1} NJ_{k-1} \left(\frac{4 \pi \sqrt{\nu\varepsilon\mu p^{m+n}}}{|q\delta c|}\right)\\
&+\sum_{\varepsilon\in \mathcal{O}^{\times +}/(\mathcal{O}^{\times})^2}\sum_{c\in p\mathfrak{n}^{\ast}}  S\left(\frac{\nu p^{m}}{q\delta},\frac{\varepsilon\mu p^{n}}{q\delta};c\right)  |N(c)|^{-1} NJ_{k-1} \left(\frac{4 \pi \sqrt{\nu\varepsilon\mu p^{m+n}}}{|q\delta c|}\right).
\end{align*}
In the last equality $(p)+\mathfrak{n}=\mathcal{O}$ is needed. The last expression is nothing but $\widetilde{c_{k}}'(\nu p^{m}, \mu p^{n} )$ and we finish the proof.
\end{proof}

Now we can deduce the following vanishing and non-vanishing relations of Hilbert Poincar\'e series.

\begin{Thm}\label{relation Poincare}
Assume $h^+=1$, $q$ is a totally positive generator of the fractional ideal $\mathfrak{c}$ and $\mu\in \mathfrak{c}^{+}$. 

(1) Let $\alpha\in\mathcal{O}^+$ be coprime to $\mu q^{-1}\mathfrak{n}$. If $P_{\mu}(z;k,\mathfrak{c},\mathfrak{n}) \equiv 0$, then $P_{\mu\alpha}(z;k,\mathfrak{c},\mathfrak{n}) \equiv 0$.

(2) Let $p$ be a totally positive prime element coprime to $\mu q^{-1}\mathfrak{n}$ and $m$ be a positive integer. If $P_{\mu}(z;k,\mathfrak{c},\mathfrak{n}) \not \equiv 0$,  then either $P_{\mu p^m}(z;k,\mathfrak{c},\mathfrak{n}) \not \equiv 0$, or $P_{\mu p^{m-1}}(z;k,\mathfrak{c},\mathfrak{n}) \not \equiv 0$ and $P_{\mu p^{m+1}}(z;k,\mathfrak{c},\mathfrak{n}) \not \equiv 0$.
\end{Thm}
\begin{proof}
We only need to prove the statements for the coefficients $\widetilde{c_{k}}(\mu, \mu)$ accordingly. Since $\widetilde{c_{k}}(\mu,\mu)=\widetilde{c_{k}}(\mu\varepsilon,\mu\varepsilon)$ for any unit $\varepsilon$, by induction, we may also assume that $\alpha=p^m$ is a power of a totally positive prime element $p$ for part (1). 

By Corollary \ref{relation coefficient} and induction on $m$, when $p$ is coprime to $\mu q^{-1}\mathfrak{n}$ and $m\geq 1$,
$$\widetilde{c_{k}}\left(\mu p^{m}, \mu p^{m}\right)=\sum_{\lambda=0}^{m} N((p))^{(k-1) \lambda} \widetilde{c_{k}}\left(\mu, \mu p^{2(m-\lambda)}\right),$$ and
$$\widetilde{c_{k}}(\mu p^{m-1}, \mu p^{m+1})=\sum_{\lambda=0}^{m-1} N((p))^{(k-1) \lambda} \widetilde{c_{k}}\left(\mu, \mu p^{2(m-\lambda)}\right).$$
It follows that
\begin{equation}\label{relation Fourier}
\widetilde{c_{k}}\left(\mu p^{m}, \mu p^{m}\right)=N((p))^{(k-1)m} \widetilde{c_{k}}(\mu, \mu)+\widetilde{c_{k}}\left(\mu p^{m-1}, \mu p^{m+1}\right).
\end{equation}

Now if $\widetilde{c_{k}}(\mu, \mu)=0$ and $\widetilde{c_{k}}\left(\mu p^{m-1}, \mu p^{m-1}\right)=0$ for any $m \geq 1$, then $P_{\mu p^{m-1}}(z;k,\mathfrak{c},\mathfrak{n})$ is identically zero and $\widetilde{c_{k}}\left(\mu p^{m-1}, \mu p^{m+1}\right)$ vanishes. It follows that $\widetilde{c_{k}}\left(\mu p^{m}, \mu p^{m}\right)=0$. By induction on $m$, it follows  that if $\widetilde{c_{k}}(\mu, \mu)=0$, then $\widetilde{c_{k}}\left(\mu p^{m}, \mu p^{m}\right)=0$, for all $m \geq 0$. This is part (1). If $\widetilde{c_{k}}(\mu, \mu) \neq 0$, it follows from (\ref{relation Fourier}) that it is not possible for both $\widetilde{c_{k}}\left(\mu p^{m}, \mu p^{m}\right)$ and $\widetilde{c_{k}}\left(\mu p^{m-1}, \mu p^{m+1}\right)$ to be zero. Note that if $\widetilde{c_{k}}\left(\mu p^{m-1}, \mu p^{m+1}\right)\neq 0$, then  $P_{\mu p^{m-1}}(z;k,\mathfrak{c},\mathfrak{n}) \not \equiv 0$ and $P_{\mu p^{m+1}}(z;k,\mathfrak{c},\mathfrak{n}) \not \equiv 0$. This gives part (2).
\end{proof}

Consequently, if $p$ coprime to $\mu q^{-1}\mathfrak{n}$ and $P_{\mu}(z;k,\mathfrak{c},\mathfrak{n}) \not \equiv 0$, then it is not possible for $P_{\mu p^{m}}(z;k,\mathfrak{c},\mathfrak{n})$ to be identically zero for two consecutive positive integers $m$. 

Very similarly, for $\Gamma_0(\mathfrak{n})$ inside $\textrm{SL}_2(\mathcal{O})$, we have the following result. We omit the analogous proof.

\begin{Thm}
Assume $h^+=1$ and $\delta$ is a totally positive generator of $\mathfrak{d}$. Let $\mu\in (\mathfrak{d}^{-1})^{+}$, $\alpha\in\mathcal{O}^{+}$ and $p$ be a totally positive prime element such that $\alpha p$ is coprime to $\mu\delta\mathfrak{n}$. Then

(1) If $P_{\mu}'(z;k,\mathfrak{n}) \equiv 0$, then $P_{\mu\alpha}'(z;k,\mathfrak{n}) \equiv 0$.

(2) If $P_{\mu}'(z;k,\mathfrak{n}) \not \equiv 0$, then for any positive integer $m$, either $P_{\mu p^m}'(z;k,\mathfrak{n}) \not \equiv 0$, or $P_{\mu p^{m-1}}'(z;k,\mathfrak{n}) \not \equiv 0$ and $P_{\mu p^{m+1}}'(z;k,\mathfrak{n}) \not \equiv 0$.
\end{Thm}

\section{Adelic Hilbert Poincar\'e Series}

As in the elliptic case \cite{R}, we relate Hilbert Poincar\'e series to Hecke operators, for which it is necessary to work on the adelic setting. We briefly recall adelic Hilbert modular forms and refer to \cite{Sh} for further details.

Let $\hat{\mathcal{O}}=\prod_\mathfrak{p}\mathcal{O}_\mathfrak{p}$ and $\mathfrak{n}$ be an integral ideal. Let $\mathcal{S}_{k}(\mathfrak{n})$ be the space of cuspidal ad\'elic Hilbert modular forms of parallel weight $k$ with level
$$K_{0}(\mathfrak{n})=\left\{\left(\begin{array}{ll}
a & b \\
c & d
\end{array}\right) \in \mathrm{GL}_{2}(\hat{\mathcal{O}}): c \in \mathfrak{n} \hat{\mathcal{O}}\right\}.$$
Let $\{ t_{j} \}_{j =1}^{h^+}$ be finite ideles such that the fractional ideals
$\mathfrak{c}_{j}:=t_{j}\mathcal{O}$
form a complete set of representatives of the narrow class group of $F$. Set $\Gamma_{j}=\Gamma_{0}\left(\mathfrak{c}_{j}, \mathfrak{n}\right)$ and $\mathcal{S}_k (\mathfrak{n}) \simeq \bigoplus_{j=1}^{h^{+}} S_k \left(\Gamma_{j}\right)$ (see \cite{Sh}), that is, adelic Hilbert modular forms are $h^+$-tuples of classical Hilbert modular forms.

Under such isomorphisms, we may write an element $f \in \mathcal{S}_k(\mathfrak{n})$ as $f=\left(f_{j}\right)$ with $f_{j} \in$ $S_k\left(\Gamma_{j}\right)$. The Petersson inner product is defined componentwisely as 
$\langle f, h\rangle=\sum_{j}\left\langle f_{j}, h_{j}\right\rangle_{\Gamma_{j}}$.
For each integral ideal $\mathfrak{m}$, assuming that $\mathfrak{m}=\mu \mathfrak{c}_j^{-1}$ with $\mu \in F^{+}$, set
$c(\mathfrak{m}, f)= N(\mathfrak{c}_j)^{-\frac{k}{2}}a_{j}(\mu)$,
where $a_{j}(\mu)$ is the $\mu$-th Fourier coefficient of $f_{j}$. This is well-defined and called the $\mathfrak{m}$-th Fourier coefficient of $f$. 

The space $\mathcal{S}_k (\mathfrak{n})$ carries a Hecke theory. For each nonzero integral ideal $\mathfrak{m}$, there exists a Hecke operator $T_\mathfrak{m}$ on $\mathcal{S}_k (\mathfrak{n})$. The Hecke operators are multiplicative, commute with each other, and away from the level $\mathfrak{n}$, they are also normal.  Under the Petersson inner product, the adjoint operator of $T_\mathfrak{m}$ is denoted by $T_\mathfrak{m}^*$. The effect of the Hecke operators on Fourier coefficients is given by
$$c\left(\mathfrak{a}, T_{\mathfrak{m}}f\right)=\sum_{\mathfrak{r} \supset \mathfrak{a}+\mathfrak{m}}\chi_0(\mathfrak{r}) N(\mathfrak{r})^{k-1} c\left(\mathfrak{a} \mathfrak{m} \mathfrak{r}^{-2}, f\right),$$
where $\chi_0(\mathfrak{r})=1$ if $\mathfrak{r}$ is coprime to $\mathfrak{n}$ and $\chi_0(\mathfrak{r})=0$ otherwise.

Now we define the Hilbert Poincar\'e series associated to a nonzero integral ideal $\mathfrak{m}$. If $\mathfrak{m}=(\mu)\mathfrak{c}_j^{-1}$, $\mu \gg 0$, then $\mu\in\mathfrak{c}_j^{+}$ and we define $P_{\mathfrak{m}}(z;k,\mathfrak{n})$ to be 
$$\frac{N(\mathfrak{c}_j)^{\frac{k+2}{2}}D^{1/2}N(\mu)^{k-1}}{(4\pi)^{(1-k)n}\Gamma(k-1)^n} P_{\mu}(z;k,\mathfrak{c}_j,\mathfrak{n})$$
in the $j$-th component and $0$ for other components. Note that
$$\langle f,P_{\mathfrak{m}}(z;k,\mathfrak{n})\rangle =\langle f_j,P_{\mathfrak{m}}(z;k,\mathfrak{n})_j\rangle_{\Gamma_j}=c(\mathfrak{m}, f),$$
so in particular $\{ P_{\mathfrak{m}}(z;k,\mathfrak{n}) : \mathfrak{m}\}$ spans $\mathcal{S}_k (\mathfrak{n})$. 

\begin{Lem}\label{action adelic Hecke}
Let $\mathfrak{m},\mathfrak{q}$ be nonzero integral ideals. Then

(1) $T_{\mathfrak{m}}^*P_{\mathfrak{q}}(z;k,\mathfrak{n})=T_{\mathfrak{q}}^*P_{\mathfrak{m}}(z;k,\mathfrak{n})$.

(2) If $\mathfrak{m}+\mathfrak{q}=\mathcal{O}$, then $T_{\mathfrak{m}}^*P_{\mathfrak{q}}(z;k,\mathfrak{n})=P_{\mathfrak{mq}}(z;k,\mathfrak{n})$.
\end{Lem}
\begin{proof} Denote $P_\mathfrak{m}=P_{\mathfrak{m}}(z;k,\mathfrak{n})$. For any $f \in \mathcal{S}_{k}(\mathfrak{n})$, we have 
\begin{align*}
\langle f,T_{\mathfrak{m}}^* P_{\mathfrak{q}}\rangle  &=\langle T_{\mathfrak{m}}f,P_{\mathfrak{q}}\rangle =c(\mathfrak{q},T_{\mathfrak{m}}f)
=\sum_{\mathfrak{m}+\mathfrak{q}\subset\mathfrak{r}}\chi_0(\mathfrak{r})N(\mathfrak{r})^{k-1}c(\mathfrak{m}\mathfrak{q}\mathfrak{r}^{-2},f),
\end{align*}
which is clearly symmetric in $\mathfrak{m}$ and $\mathfrak{q}$. Therefore, $\langle f, T_{\mathfrak{m}}^* P_{\mathfrak{q}}\rangle=\langle f, T_{\mathfrak{q}}^* P_{\mathfrak{m}}\rangle$ and part (1) follows from this.

When $\mathfrak{m}+\mathfrak{q}=\mathcal{O}$, the above formula gives $\langle f, T_{\mathfrak{m}}^* P_{\mathfrak{q}}\rangle= c(\mathfrak{m}\mathfrak{q},f)=\langle f,P_{\mathfrak{mq}}\rangle$, and part (2) follows.
\end{proof}

\begin{Thm}\label{Adelic-Vanishing} Let $\mathfrak{m}$ and $\mathfrak{q}$ be nonzero integral ideals.

(1) If $\mathfrak{m}+\mathfrak{q}=\mathcal{O}$ and $P_{\mathfrak{m}}(z;k,\mathfrak{n}) \equiv 0$, then $P_{\mathfrak{mq}}(z;k,\mathfrak{n}) \equiv 0$.

(2) $P_{\mathfrak{m}}(z;k,\mathfrak{n}) \equiv 0$ if and only if  $T_{\mathfrak{m}}=0$.
\end{Thm}
\begin{proof}
Part (1) follows immediately from part (2) of Lemma \ref{action adelic Hecke}. As for part (2), if  $P_{\mathfrak{m}}(z;k,\mathfrak{n}) \equiv 0$, then $T_{\mathfrak{m}}^*P_\mathfrak{q}(z;k,\mathfrak{n})=0$ for any nonzero integral ideal $\mathfrak{q}$ by part (1) of Lemma \ref{action adelic Hecke}. Since Hilbert Poincar\'e series span the whole space, we have $T_\mathfrak{m}^*=0$ and hence $T_\mathfrak{m}=0$. Conversely, if $T_{\mathfrak{m}}=0$, we have $T_{\mathfrak{m}}^*=0$ and then 
$P_{\mathfrak{m}}(z;k,\mathfrak{n})=T_{\mathfrak{m}}^* P_{\mathcal{O}}(z;k,\mathfrak{n})=0$ as desired.
\end{proof} 

More generally, we see that Poincar\'e series and Hecke operators possess same linear relations.

\begin{Cor}
For finitely many nonzero integral ideals $\mathfrak{m}_i$ and complex numbers $\lambda_i$, we have $\sum_i \lambda_i P_{\mathfrak{m}_i}(z;k,\mathfrak{n}) \equiv 0$ if and only if  $\sum_i \overline{\lambda_i} T_{\mathfrak{m}_i}=0$. 
\end{Cor}
\begin{proof}
The proof is the same as that of part (2) of Theorem \ref{Adelic-Vanishing}.
\end{proof}


\begin{thebibliography}{10}

\bibitem{B} A. Balog,  K. Ono, The Chebotarev density theorem in short intervals and some questions of Serre, Journal of Number Theory, 2001, 91(2), 356-371.
\bibitem{11} R. W. Bruggeman, R. J. Miatello, Estimates of Kloosterman sums for groups of real rank one, Duke Mathematical Journal, 1995, 80 (1), 105-137.
\bibitem{F} E. Freitag, Hilbert modular forms, Springer, 1990.
\bibitem{G} E. Gaigalas, Poincar\'e series that do not identically vanish, Lietuvos Matematikos Rinkinys, 1986, 26 (3), 431-434. 
\bibitem{Garr} P. Garrett, Holomorphic Hilbert modular forms, Wadsworth and Books/Cole, 1990. 
\bibitem{18} S. S. Gelbart, Automorphic Forms on Adele Groups, Princeton University Press, 1975.
\bibitem{15} G. Harcos, G. Karolyi, Selberg's identity for Kloosterman sums, https://users.renyi.hu/\textasciitilde gharcos /selberg\underline{ }identity.pdf
\bibitem{I} H. Iwaniec, Topics in classical modular forms, American Mathematical Society, 1997. 
\bibitem{K1} E. Kowalski, O. Robert and J. Wu, Small gaps in coefficients of $L$-functions and $\mathfrak{B}$-free numbers in short intervals, Revista Matem\'atica Iberoamericana, 2007, 23(1), 281-326.
\bibitem{K} E. Kowalski, A. Saha and J. Tsimerman, A note on Fourier coefficients of Poincar\'e series, Mathematika, 2011, 57(1), 31-40. 
\bibitem{9} M. Kumari, Non-vanishing of Hilbert Poincar\'e series, Journal of Mathematical Analysis and Applications, 2018, 466(2), 1476-1485. 
\bibitem{L} D. H. Lehmer, The vanishing of Ramanujan's function $\tau (n)$, Duke Mathematical Journal, 1947, 14, 429-433. 
\bibitem{Lehn} J. Lehner, On the nonvanishing of Poincar\'e series, Proceedings of the Edinburgh Mathematical Society, 1980, 23(2), 225-228.
\bibitem{10} W. Luo, Poincar\'e series and Hilbert modular forms, The Ramanujan Journal, 2003, 7, 129-140. 

\bibitem{7} C. J. Mozzochi, On the nonvanishing of Poincar\'e series, Proceedings of the Edinburgh Mathematical Society, 1989, 32(2), 133-137.
\bibitem{O} F. Olver, D. Lozier, R. Boisvert and C. Clark, The NIST handbook of mathematical functions, Cambridge University Press, 2010.
\bibitem{20} K. Ono, Unearthing the visions of a master: harmonic Maass forms and number theory, Current Developments in Mathematics, 2009, 347-454.
\bibitem{R} R. A. Rankin, The vanishing of Poincar\'e series, Proceedings of the Edinburgh Mathematical Society, 1980, 23(2), 151-161.
\bibitem{19} J. P. Serre, Quelques applications du th\'eor\`eme de densit\'e de Chebotarev, Publications Math\'ematiques de l'Institut des Hautes \'Etudes Scientifiques, 1981, 54, 123-201.
\bibitem{Sh} G. Shimura, The special values of the zeta functions associated with Hilbert modular forms, Duke Mathematical Journal, 1978, 45(3), 637-679.

\end{thebibliography}
\end{document}